\documentclass[12pt]{article}
\usepackage{amsmath, amsfonts, amsthm, amssymb}
\usepackage{amstext}
\usepackage{amssymb,mathrsfs}
\newcommand{\R}{\mathbb R}
\numberwithin{equation}{section}
\newtheorem{theorem}{Theorem}[section]
\newtheorem{proposition}[theorem]{Proposition}
\newtheorem{remark}[theorem]{Remark}
\newtheorem{lemma}[theorem]{Lemma}

\newtheorem{definition}[theorem]{Definition}

\begin{document}

\title{ $L^{p}$-solutions of the Navier-Stokes equation with fractional Brownian noise}

\author{Benedetta Ferrario\thanks{Corresponding Author:
Dipartimento di Matematica, Universit\`{a} di Pavia,
Italy. E-mail: \textsl{benedetta.ferrario@unipv.it}, Phone (+39) 0382 985655}, Christian
Olivera\thanks{Departamento de Matem\'{a}tica, Universidade Estadual de
Campinas, Brazil. E-mail: \textsl{colivera@ime.unicamp.br} }}

\date{\today}

\maketitle

\textit{Key words and phrases:
stochastic partial differential equations, Navier-Stokes equations, mild solution, fractional Brownian motion.}

\vspace{0.3cm} \noindent {\bf MSC2010 subject classification:} 60H15, 
 76D06, 76M35 , 35Q30.

%
%%%%%%%%%%%%%%%%%%%%%%%%%%%%%%%%%%%%%%%%%%%%%%%%%%%%%%%%%%%%%%%%%%%%%%%%%%%%
\begin{abstract}
We study the Navier-Stokes equations  
on a smooth bounded domain $D\subset \mathbb R^d$ ($d=2$ or 3), 
under the effect of an additive fractional Brownian noise. 
We show local existence and uniqueness of a  mild $L^p$-solution for $p>d$. 
\end{abstract}

%%%%%%%%%%%%%%%%%%%%%%%%%%%%%%%%%%%%%%%%%%%%%%%%%%%%%%%%%%%%%%%%%%%%%%%%%%%%
%
\maketitle

%%%%%%%%%%%%%%%%%%%%%%%%%%%%%%%%%%%%%%%%%%%%%%%%%%%%
\section {Introduction} \label{Intro}

The Navier-Stokes equations %(briefly NSE) 
have been derived, more than one century ago,
by the engineer C. L. Navier to describe the motion of an incompressible Newtonian fluid.
Later, they have been reformulated by the mathematician- physicist G. H. Stokes. Since that
time, these equations continue to attract a great deal of attention due to their mathematical
and physical importance. In a seminal paper \cite{Leray}, Leray proved the global existence 
of a weak solution with finite energy. It is well known that weak solutions are unique and regular in two
spatial dimensions. In three dimensions, however, the question of regularity and
uniqueness of weak solutions is an outstanding open problem in mathematical fluid
mechanics, we refer to  excellent monographs \cite{Lema}, \cite{lion1} and  \cite{Soh}. 

More recently, stochastic versions of the  Navier-Stokes equations have been considered 
in the literature; 
first by introducing a stochastic forcing term which comes from a
 Brownian motion (see, e.g., the first results in  \cite{BT,VF}).
The addition of the white noise driven term to the basic
 governing equations is natural for both practical and theoretical applications  to take
 into account for numerical and empirical
uncertainties, and have been proposed as a model for turbulence. 
 Later on other kinds of noises have been studied. 

In this paper we  consider  the the  Navier-Stokes equations with a stochastic forcing term modelled by a 
fractional Brownian motion 
\begin{equation}\label{Navier}
 \left \{
\begin{aligned}
     &\partial_t u(t, x) =\big(\nu  \Delta u(t, x) - [u(t, x)\cdot \nabla] u(t, x)  - \nabla \pi(t, x)\big)dt +    \Phi \partial_{t} W^{{\mathcal H}}(t,x) 
    \\& \text{div } u(t, x)=0
    \\    &u(0,x)=  u_{0}(x)
\end{aligned}
\right .
\end{equation}
We fix a smooth bounded domain $D\subset \mathbb R^d$ ($d=2$ or $3$)
and consider the homogeneous Dirichlet  boundary condition. In the equation above, 
$u(t, x)\in \R^{d}$ denotes the vector velocity field at time $t$ and position $x\in D$, $\pi(t, x)$ 
denotes the pressure field,  $\nu > 0$ is the viscosity coefficient. 
In the random forcing term there appears a Hilbert space-valued  cylindrical fractional Brownian motion $W^{{\mathcal H}}$
with Hurst parameter ${\mathcal H}\in (0,1)$ and a linear bounded operator $\Phi$ to characterize the spatial covariance of the noise.

When ${\mathcal H}=\frac{1}{2}$, i.e $W^{\frac 12}$ is the Wiener process,  there is a large amount of literature on the stochastic
Navier-Stokes equation \eqref{Navier} and its abstract setting.  For an overview of the known results, recent developments, as well as further references,
we refer to   \cite{Bene}, \cite{Fla}, \cite{K}  and  \cite{VF}.
On the other hand, when  $0<{\mathcal H}< 1$ there are results by 
Fang,   Sundar and   Viens;
in \cite{Fang} they prove when $d=2$ the existence of a unique 
global solution which is 
$L^{4}$ in time and in space 
by assuming that  the Hurst parameter ${\mathcal H}$ satisfies a condition involving the regularity of  $\Phi$.

Our aim is to  deal with 
$L^{p}$-solutions of the Navier-Stokes systems \eqref{Navier} for $p>d$.
Our approach  to study $L^{p}$-solutions  is based on the concept of mild solution as in \cite{Fang}; but we deal with 
dimension $d=2$ as well as with $d=3$ and any $p>d$.

We shall prove a local existence and  uniqueness result. Some remarks on global solutions will also be given.
Let us recall also  that 
results on the local existence of mild $L^{p}$-solutions  in the deterministic setting
were established in the papers \cite{Giga}, \cite{Giga2},
\cite{Kato},\cite{Kato2}, \cite{Kato3}, \cite{W}. 

In more details, in Section \ref{sec-setting} we shall introduce the mathematical setting,
in Section \ref{sec-linear} we shall deal with the linear problem and in Section \ref{sec-ex} we shall prove our main result.

%%%%%%%%%%%%%%%%%%%%%%%%%%%%%%%%%%%%%%%%%%%%%%%%%%%%

\section{Functional setting}\label{sec-setting}
In this section we introduce the functional setting to rewrite system \eqref{Navier} in abstract form.
%%%%%%%%%%%%%%%%%%
\subsection{The functional spaces}
Let $D$ be a bounded domain 
in $\R^{d}$ ($d\ge 2$) with smooth boundary $\partial D$. For $1\le p<\infty$ we denote 
\[
L_{\sigma}^{p}= \text{  the  closure   in } [L^{p}(D)]^d  \text{ of  }\{ u\in [C_{0}^{\infty}(D)]^d, \ div \ u=0\}
\]
and
\[
G^{p}= \{ \nabla q,   q\in W^{1,p}(D) \}.
\]

We then have the following Helmholtz decomposition 
\[
[L^{p}(D)]^d= L_{\sigma}^{p} \oplus G^{p}, 
\]
where the notation $\oplus$ stands for the direct sum. In the case $p=2$ 
 the sum above reduces to the orthogonal decomposition and $L_{\sigma}^2$
is a separable Hilbert space, whose scalar product is denoted by $(\cdot,\cdot)$.

%%%%%%%%%%%%%%%%%%
\subsection{The Stokes operator}
Let us  recall some results on the Stokes operator (see, e.g., \cite{Soh}). 

Now we fix $p$. Let $P$ be the continuous projection  from $[L^{p}(D)]^d$ onto  $L_{\sigma}^{p}$
and let $\Delta$ be
the Laplace operator in $L^p$  with zero boundary condition, so that $D(\Delta)=\{u \in [W^{2,p}(D)]^d: u|_{\partial D}=0\}$.

Now, we define the Stokes operator $A$ in $L_{\sigma}^{p}$ by $A=-P\Delta$
with domain $H^{2,p}:=L_{\sigma}^{p}\cap D(\Delta)$.  The operator $-A$ generates  a bounded analytic semigroup 
$\{S(t)\}_{t\ge 0}$ of class $C_{0}$ in $L_\sigma^p$.

In particular, for $p=2$ we set $H^2=H^{2,2}$ and the Stokes operator $A : H^{2}\rightarrow L_{\sigma}^2$
is an isomorphism, the inverse operator 
$A^{-1}$   is self-adjoint  and  compact in $L_{\sigma}^{2}$. Thus, there exists an
orthonormal basis $\{e_{j}\}_{j=1}^\infty\subset H^{2}$ of $L^{2}_{\sigma}$
  consisting of the eingenfunctions of $A^{-1}$ and such that 
the sequence of eigenvalues $\{\lambda_{j}^{-1}\}_{j=1}^\infty$, with $\lambda_{j}>0$, converges to zero as $j \to \infty$. In particular,
$\lambda_j$ behaves as $j^{\frac 2d}$ for  $j \to \infty$. Then, 
$\{e_{j}\}_j$ is also the sequence of eingenfunctions of $A$ corresponding to the eigenvalues 
$\{\lambda_{j}\}_j$. Moreover $A$ a is positive, selfadjoint 
 and densely defined operator  in  $L_{\sigma}^{2}$. Using the spectral decomposition, 
we construct positive and negative  fractional 
power operators $A^{\beta}$, $\beta\in \R$. For $\beta\ge 0 $ we have the following representation for  
$(A^{\beta}, D(A^{\beta}))$ as a linear operator in $L_{\sigma}^{2}$
\[
D(A^{\beta})=
\big\{  v\in  L_{\sigma}^{2}:\ \|v\|^2_{D(A^\beta)}=
\sum_{j=1}^\infty \lambda_{j}^{2\beta} |(v,e_{j})|^{2} < \infty   \big\},
\]
\[
A^\beta v= \sum_{j=1}^\infty \lambda_{j}^\beta (v,e_{j}) e_{j}. 
\]
For negative exponents, we get the dual space: $D(A^{-\beta})=(D(A^\beta))'$. We set $H^s=D(A^{\frac s2})$.
Let us point out that the operator $A^{-\beta}$ is an Hilbert-Schmidt
operator in $L^2_\sigma$ for any $\beta>\frac d4$; indeed, denoting by 
$\|\cdot \|_{\gamma(L^2_\sigma,L^2_\sigma)} $ the Hilbert-Schmidt
norm, we have
\[
\|A^{-\beta}\|_{\gamma(L^2_\sigma,L^2_\sigma)}^2 
:= 
\sum_{j=1}^\infty \|A^{-\beta}e_j\|^2_{L^2_\sigma}
=
\sum_{j=1}^\infty \lambda_j^{-2\beta}
\sim  
\sum_{j=1}^\infty j^{-2\beta\frac 2d}
\]
and the latter series in convergent for $2\beta\frac 2d>1$.

We also  recall (see, e.g., \cite{W}) that for any $t>0$ we have
\begin{equation}\label{stimaSpr}
\| S(t) u \|_{{L^p_\sigma}}\le  \frac{M}{t^{\frac{d}{2}(\frac 1r-\frac 1p) } } \| u \|_{L^r_\sigma} \  \text{ for } \ 1< r\le p < \infty
\end{equation}
\begin{equation}\label{stimaASp}
\|A^{\alpha} S(t) u \|_{L^r_\sigma}\le  \frac{M}{t^{\alpha}} \| u \|_{L^r_\sigma} \  \text{ for } \ 1<  r < \infty, \, \alpha >0
\end{equation}
for any $u \in L^r_\sigma$, 
where $M$ denotes different constants depending on the parameters.
Moreover we have the following result on the Hilbert-Schmidt norm of the semigroup, that we shall use later on.
What is important is the behaviour for $t$ close to $0$, let us say for $t \in (0,1)$.
\begin{lemma}\label{qsmall}
We have 
\[
\|S(t)\|_{\gamma(H^{\frac d2};L^2_\sigma)}
\le 
M(2-\ln t) \qquad\forall t\in (0,1)
\]
and for $q< \frac d2$
\begin{equation}\label{qd2}
\|S(t)\|_{\gamma(H^q;L^2_\sigma)}
\le 
\frac M{t^{\frac d4-\frac q2}} \qquad\forall t>0
\end{equation}
\end{lemma}
\begin{proof}
The Hilbert-Schmidt norm of the semigroup can be computed. 
Recall that  $\{\frac {e_j}{\lambda_j^{q/2}}\}_j$ is an orthonormal basis of $H^q$. Thus
\[
\|S(t)\|_{\gamma(H^q,L^2_\sigma)}^2
=
\sum_{j=1}^\infty \|S(t) \frac {e_j}{\lambda_j^{q/2}}\|_{L^2_\sigma}^2
=
\sum_{j=1}^\infty \frac 1{\lambda_j^{q}} \|e^{-\lambda_j t} e_j\|_{L^2_\sigma}^2
=
\sum_{j=1}^\infty \frac {e^{-2\lambda_j t}}{\lambda_j^{q}}.
\]
Since $\lambda_j \sim j^{\frac 2d}$ as $j\to \infty$, we estimate
\[
\|S(t)\|_{\gamma(H^q,L^2_\sigma)}^2\le C
\sum_{j=1}^\infty \frac {e^{-2j^{\frac 2d} t}}{j^{\frac{2q}d}}.
\]
Therefore we analyse the series $s_q(t)=\displaystyle \sum_{j=1}^\infty \frac {e^{-2j^{\frac 2d} t}}{j^{\frac{2q}d}}$.
Let us consider different values of the parameter $q$.
\\$\bullet$ 
When $q=\frac d2$ 
 the series  becomes 
\[
s_{\frac d2}(t)=\sum_{j=1}^\infty j^{-1} e^{-2j^{\frac 2d} t}
=
e^{-2t}+\sum_{j=2}^\infty j^{-1} e^{-2j^{\frac 2d} t}
\le
e^{-2t}+\int_1^\infty \frac 1x e^{-2x^{\frac 2d} t} dx.
\]
The integral
is computed by means of the change of variable $x=y^d t^{-\frac d2}$ so to get
\[
\int_1^\infty \frac 1x e^{-2x^{\frac 2d} t} \text{d}x=\int_{\sqrt t}^\infty \frac dy e^{-2y^2}\text{d}y .
\]
Hence, for $t \in (0,1)$  we get
\[
s_{\frac d2}(t)\le e^{-2t}+d \int_{\sqrt t}^1 \frac 1y \text{d}y
+\int_1^\infty e^{-2y^2}\text{d}y
\le 1-\frac d2 \ln t+C.
\]
$\bullet$ 
When $0\le q < \frac d2$ then  the sequence of the addends
is monotone decreasing and 
therefore we estimate the series by an integral:
\[
\sum_{j=1}^\infty \frac {e^{-2j^{\frac 2d} t}}{j^{\frac{2q}d}}
\le
\int_0^\infty \frac {e^{-2x^{\frac 2d} t}}{x^{\frac{2q}d}}dx .
\]
Again, by the change of variable $x=y^d t^{-\frac d2}$ we calculate the integral and get
\[
\sum_{j=1}^\infty \frac {e^{-2j^{\frac 2d} t}}{j^{\frac{2q}d}}\le 
t^{q-\frac d2} d\int_0^\infty y^{d-2q-1} e^{-2y^2}\text{d}y .
\]
The latter integral is convergent since $d-2q-1>-1$ by the assumption that $q<\frac d2$. 
Hence we get the bound \eqref{qd2} for the Hilbert-Schmidt norm of $S(t)$.
\\$\bullet$ 
When $q<0$  the sequence of the addends in the series $s_q(t)$ 
is first increasing and then decreasing. 
 Let us notice that $t\mapsto s_q(t)$ (defined for $t>0$) 
is a continuous decreasing  positive function converging to $0$ as $t \to +\infty$.
Hence to estimate it for $t\to 0^+$ it is enough to get an estimate over a sequence $t_n\to 0^+$. 
We choose this sequence in such a way that the maximal value of the function 
$a_t(x):= x^{-\frac {2q}d} e^{-2x^{\frac 2d} t}$ (defined for $x>0$) is attained at the 
integer value $n=(-\frac q{2t_n})^{\frac d2} \in \mathbb N$. 
In this way we can estimate the series by means of an integral:
\[\begin{split}
s_q(t_n)\equiv \sum_{j=1}^\infty a_{t_n}(j)&\le
\int_1^n a_{t_n}(x)\text{d}x+ a_{t_n}(n)+ \int_n^\infty a_{t_n}(x)\text{d}x
\\&
= \int_1^\infty x^{-\frac{ 2q}d} e^{-2x^{\frac 2d} t_n} \text{d}x+n^{-\frac {2q}d} e^{-2 n^{\frac 2d} t_n}
\\&
\le d\Big(\int_0^\infty y^{d-1-2q}e^{-2y^2}\text{d}y\Big)\ t_n ^{q-\frac d2}
+C_q t_n^q
\end{split}\]
where we have computed the integral by means of the change of variable $x=y^d t_n^{-\frac d2}$ as before.
Hence,  we get that 
\[
s_q(t_n)\le \tilde C t_n^{q-\frac d2} \;\text{ for any }n
\]
and therefore for $t \to 0^+$
\[
s_q(t)\le\frac C{t^{\frac d2-q}}.
\]
This proves \eqref{qd2} when $q<0$.
\end{proof}

%%%%%%%%%%%%%%%%%
\subsection{The bilinear term}
Let us define the nonlinear term  by $B(u,v)=-P[(u\cdot \nabla)v]$. 
Following \cite{Soh},
this is first defined on smooth  divergence free vectors fields with compact support and one proves by integration by parts that
\begin{equation}
\langle B(u,v), z\rangle=-\langle B(u,z), v\rangle, \qquad \langle B(u,v), v\rangle=0
\end{equation} 
Then one specifies that $B$ is continuous with respect to suitable topologies. In particular, H\"older inequality
provides
\[
 \|B(u,v)\|_{H^{-1}}\le \|u\|_{L^4_\sigma} \|v\|_{L^4_\sigma}
\]
and thus $B:L^4_\sigma\times L^4_\sigma \to H^{-1}$ is continuous.

Since $u$ is  a divergence free vector field, we also have the representation
$B(u,v)= -P[ \text{div}\ (u\otimes v )]$ which will be useful later on (again this holds for smooth entries and then is extended for $u$ and $v$ suitably regular).

For short we shall write $B(u)$ instead of $B(u,u)$.

%%%%%%%%%%%%%%%%%%%%%
\subsection{Fractional Brownian motion}

First, we recall that  a real fractional Brownian
motion  (fBm) $\{B^{{\mathcal H}}(t)\}_{t\in [0,T]}$ with Hurst  parameter ${\mathcal H}\in(0,1)$ is 
 a centered Gaussian  process with covariance function 
\begin{equation}
\label{cov}
\mathbb{E}[ B ^{{\mathcal H}}(t) B ^{{\mathcal H}}(s)] :=R_{{\mathcal H}}(t,s)= \frac{1}{2} ( t^{2{\mathcal H}} + s^{2{\mathcal H}}- \vert t-s\vert ^{2{\mathcal H}}), \hskip0.5cm s,t \in [0,T].
\end{equation}
For more details see \cite{N}.

We are interested in the infinite dimensional fractional Brownian motion.
We consider the separable Hilbert space $L^{2}_{\sigma}$ and its  orthonormal basis
$\{e_j\}_{j=1}^\infty$. Then we define
\begin{equation}\label{cfBm}
W^{{\mathcal H}}(t)=\sum_{j=1}^{\infty} e_{j} \beta_{j}^{{\mathcal H}}(t)
\end{equation}
where $\{ \beta_{j}^{{\mathcal H}} \}_j$ is a family of independent  real fBm's
defined on a complete filtered  probability space $(\Omega, \mathbb F, \{\mathbb F_t\}_t,  \mathbb P)$.
This is the so called $L^{2}_{\sigma}$-cylindrical fractional Brownian motion. 
Moreover we consider a linear 
operator $\Phi$ defined in $L^2_\sigma$. Notice that the series in
\eqref{cfBm} does not converge in $L^2_\sigma$.

We need to define the integral of the form $\int_{0}^{t} S(t-s) \Phi dW^{{\mathcal H}}(s)$, 
appearing in the definition of mild solution; we will analyze this stochastic integral 
 in Section \ref{sec-linear}.

%%%%%%%%%%%%%%%%%%%%
\subsection{Abstract equation}
Applying the projection operator $P$ to \eqref{Navier}
we get rid of the pressure term; setting $\nu=1$, equation \eqref{Navier} becomes
\begin{equation}\label{eq-abst}
\begin{cases}
du(t) + Au(t)\ dt= B(u(t)) \ dt + \Phi \text{d} W^{{\mathcal H}}(t), & t>0 
\\ u(0)=u_{0}
\end{cases}
\end{equation}
We consider its mild solution.
\begin{definition}  A measurable function $u:\Omega\times [0,T]\rightarrow L_{\sigma}^{p}$
is a mild $L^p$-solution of  equation \eqref{eq-abst} if

$\bullet$ $u \in C([0,T]; L^{p}_\sigma)$, $\mathbb P$-a.s.

$\bullet$ for all $t \in (0,T]$, we have
\begin{equation}\label{mild}
u(t)= S(t)u_{0}+  \ \int_{0}^{t} S(t-s)B(u(s)) \  ds +  \int_{0}^{t} S(t-s) \Phi  \text{d}W^{{\mathcal H}}(s)
 \end{equation}

$\mathbb P$-a.s.
\end{definition}

%%%%%%%%%%%%%%%%%%%%%
\section{The linear equation}\label{sec-linear}
Now we  consider the linear problem associated to the Navier-Stokes equation \eqref{eq-abst}, that is 
\begin{equation}\label{eq-z}
 d z(t)+   Az(t) \ dt = \Phi  dW^{{\mathcal H}}(t)
 \end{equation}
When the      initial condition is $z(0)=0$,
its mild solution is the stochastic convolution
\begin{equation}\label{z-proc}
z(t)=   \int_{0}^{t} S(t-s) \Phi \  dW^{{\mathcal H}}(s).
\end{equation}
To analyze its regularity we appeal to the following result.
\begin{proposition}\label{pro-gen}
Let $0<{\mathcal H}<1$.
 \\
If there exist $\lambda, \alpha \ge 0$ such that 
\begin{equation}\label{SPhi}
\|S(t) \Phi \|_{\gamma(L^2_\sigma,L^2_\sigma)}\le \frac C{t^\lambda} \qquad\forall t>0
\end{equation}
and 
\begin{equation}
\lambda+\frac \alpha2<{\mathcal H}
\end{equation}
then $z$ has a version which belongs to $C([0,T];H^\alpha)$.
\end{proposition}
\begin{proof}
This is a well known result for ${\mathcal H}=\frac 12$.
Moreover, the case ${\mathcal H}<\frac 12$  is proved in  Theorem 11.11 of \cite{PDDM2006}  and the case ${\mathcal H}>\frac 12$ in 
Corollary 3.1 of \cite{DPDM2002}, by assuming that the semigroup $\{S(t)\}_t$ is analytic.
\end{proof}

Now we use this result with $\alpha=d(\frac 12-\frac 1p)$ for $p>2$; 
by means of the Sobolev embedding $H^{d(\frac 12-\frac 1p)}(D)\subset L^p(D)$,
this provides that  $z$ has a version which belongs to $C([0,T];L^p_{\sigma})$.
 
We have our regularity result for the stochastic convolution by assuming that 
 $\Phi \in \mathcal L(L^2_\sigma;H^q)$ for some $q \in \mathbb R$, 
as e.g. when $\Phi= A^{-\frac q2}$.
%PROP
\begin{proposition}\label{pro-z}
Let $0<{\mathcal H}<1$, $2<p<\infty$ and $\Phi \in \mathcal L(L^2_\sigma,H^q)$ for some $q \in \mathbb R$.
If the parameters fulfil
\begin{equation}\label{cond-dpq}
 \frac d2(1-\frac 1p)-\frac q2<{\mathcal H}
\end{equation}
then  the process $z$ given by \eqref{z-proc}  has a version which belongs to $C([0,T];H^{d(\frac 12-\frac 1p)})$.
By Sobolev embedding this version is in $C([0,T];L^p_\sigma)$ too.
\end{proposition}
\begin{proof}
According to Proposition \ref{pro-gen} we have to  estimate the Hilbert-Schmidt norm 
of the operator $S(t) \Phi$.  We recall that the product of two linear operators is  Hilbert-Schmidt if at least one 
of them is  Hilbert-Schmidt.

Bearing in mind  Lemma \ref{qsmall}, when $q<\frac d2$  
we get
\begin{equation}\label{norm-operators2}
\|S(t) \Phi \|_{\gamma(L^2_\sigma,L^2_\sigma)}\le
\|S(t)\|_{\gamma(H^q,L^2_\sigma)} \|\Phi \|_{\mathcal L(L^2_\sigma,H^q)}
\le \frac C{t^{\frac d4-\frac q2}} 
\end{equation}
and when $q=\frac d2$ we get
\begin{equation}\label{norm-operators3}
\|S(t) \Phi \|_{\gamma(L^2_\sigma,L^2_\sigma)}\le
\|S(t)\|_{\gamma(H^{\frac d2},L^2_\sigma)} \|\Phi \|_{\mathcal L(L^2_\sigma,H^{\frac d2})}
\le \frac C{t^a}
\end{equation}
for any $a>0$ (here the constant depends also on $a$).
Therefore when $q<\frac d2$   we choose $\lambda=\frac d4-\frac q2$, 
$\alpha=d(\frac 12-\frac 1p)$  and condition $\lambda+\frac \alpha 2<{\mathcal H}$ becomes 
\eqref{cond-dpq}; when  $q=\frac d2$ we choose $\lambda=a$, 
$\alpha=d(\frac 12-\frac 1p)$ and 
since $a$ is arbitrarily small we get again 
\eqref{cond-dpq}.

Otherwise, 
when  $q>\frac d2$ we have that $\Phi$ is a Hilbert-Schmidt operator in $L^2_\sigma$ 
(since $ \|\Phi \|_{\gamma(L^2_\sigma,L^2_\sigma)}\le
\|A^{-\frac q2} \|_{\gamma(L^2_\sigma,L^2_\sigma)} \|A^{\frac q2}\Phi \|_{\mathcal L(L^2_\sigma,L^2_\sigma)}$) and we estimate 
\begin{equation}\label{norm-operators}
\|S(t) \Phi \|_{\gamma(L^2_\sigma,L^2_\sigma)}\le
\|S(t)\|_{\mathcal L(L^2_\sigma,L^2_\sigma)} \|\Phi \|_{\gamma(L^2_\sigma,L^2_\sigma)}\le
C
\end{equation}
for all $t\ge 0$.
Actually we can prove something more; 
we write $A^{\frac 12(q-\frac d2)}=A^{\varepsilon}  A^{-\frac d4 -\varepsilon} A^{\frac q2}$ and 
for any $\varepsilon>0$ we have
\[\begin{split}
\|S(t) A^{\frac 12(q-\frac d2)}\Phi\|_{\gamma(L^2_\sigma,L^2_\sigma)}
&\le
\|A^\varepsilon S(t)\|_{\mathcal L(L^2_\sigma,L^2_\sigma)}
\|A^{-\frac d4 -\varepsilon}\|_{\gamma(L^2_\sigma,L^2_\sigma)}
\|A^{\frac q2}\|_{\mathcal L(H^q,L^2_\sigma)} \|\Phi\|_{\mathcal L(L^2_\sigma;H^q)}
\\&
\le \frac M{t^\varepsilon}
\end{split}
\]
According to Proposition \ref{pro-gen}, choosing 
$\gamma=\varepsilon$ and $\alpha=d(\frac 12-\frac 1p)-(q-\frac d2)$
we obtain that the process
\[
\int_0^t  S(t-s) A^{\frac 12(q-\frac d2)} \Phi \  dW^{{\mathcal H}}(s), \quad t \in [0,T]
\]
has a $C([0,T];H^{d(\frac 12-\frac 1p)-(q-\frac d2)})$-valued version 
if
\[
\varepsilon+\frac 12[ d(\frac 12-\frac 1p)-(q-\frac d2)] <{\mathcal H}<1
\]
i.e. choosing $\varepsilon$ very small, if
\[
\frac d2 (1-\frac 1p)-\frac q2<{\mathcal H}<1.
\]
Since $S(t)$ and $A^{\frac 12(q-\frac d2)}$ commute, we get as usual that the result holds for the process
$A^{\frac 12(q-\frac d2)}z$. Therefore $z$ has a $C([0,T];H^{d(\frac 12-\frac 1p)})$-version.
Actually this holds when $\alpha=d(\frac 12-\frac 1p)-(q-\frac d2)\ge 0$, that is when
$q\le d(1-\frac 1p)$. For larger values of $q$ the regularising effect of the operator $\Phi$ is even better and the result holds true for any $0<{\mathcal H}<1$.
\end{proof}

%REMARK
\begin{remark} 
Instead of appealing to the Sobolev embedding $H^{d(\frac 12-\frac 1p)}\subset L^p_\sigma$, 
we could look directly for an $L^p$-mild solution $z$, that is a process with $\mathbb P$-a.e. path in 
 $C([0,T];L^p_\sigma)$. 
 Let us check if this approach would be better.
 
 There are results providing the regularity in Banach spaces; see e.g.  
 Corollary 4.4. in the paper  \cite{Co} by   \v{C}oupek,  Maslowski,
 and Ondrej\'at. They involve the 
$\gamma$-radonifying norm instead of the Hilbert-Schmidt norm (see,
 e.g., \cite{vN} for the definition of these norms).
However 
 the estimate of the $\gamma$-radonifying norm of $S(t) \Phi $ is not trivial. 
The estimates involved lead anyway to work in a Hilbert space
setting. Let us provide some details about this fact.

 According to \cite{Co},  assuming  $\frac 12 <{\mathcal H}<1$  and $1\le p{\mathcal H}<\infty$ one should verify that  
 there exists $\lambda \in [0,{\mathcal H})$ such that 
\[
\|S(t) \Phi \|_{\gamma(L^2_\sigma,L^p_\sigma)}\le \frac C{t^\lambda} \qquad\forall t>0
\]
\\
Given $\Phi \in \mathcal L(L^2;H^q)$
 we just have to estimate  the $\gamma(H^q,L^p_\sigma)$-norm of $S(t)$, since
\[
\|S(t) \Phi \|_{\gamma(L^2_\sigma,L^p_\sigma)}\le
\|S(t)\|_{\gamma(H^q,L^p_\sigma)}  \|\Phi \|_{\mathcal L(L^2_\sigma,H^q)}.
\]
The $\gamma(H^q,L^p_\sigma)$-norm of $S(t)$ is equivalent
to 
\[
\left[\int_D \Big(\sum_{j=1}^\infty 
|S(t) \frac {e_j(x)}{\lambda_j^{q/2}}|^2\Big)^{\frac p2}dx \right]^{1/p}
\]
since $\{\frac {e_j}{\lambda_j^{q/2}}\}_j$ is an orthonormal basis of $H^q$.

Therefore, we estimate the integral. Let us do it for $p \in 2\mathbb N$.
We have
\[\begin{split}
\int_D \Big(\sum_{j=1}^\infty 
|S(t) \frac {e_j(x)}{\lambda_j^{q/2}}|^2\Big)^{\frac p2}dx
&=
\int_D \Big(\sum_{j=1}^\infty \lambda_j^{-q} e^{-2\lambda_j t}
|e_j(x)|^2\Big)^{\frac p2}dx
\\&=  
\int_D \Pi_{n=1}^{p/2}(\sum_{j_n=1}^\infty \lambda_{j_n}^{-q} e^{-2\lambda_{j_n} t}
|e_{j_n}(x)|^2) dx
\end{split}\]
Using the H\"older inequality, we get 
\[
\int_D |e_{j_1}(x)|^2 |e_{j_2}(x)|^2 \cdots |e_{j_{p/2}}(x)|^2 dx
\le 
\|e_{j_1}\|_{L^p}^2 \|e_{j_2}\|_{L^p}^2 \cdots \|e_{j_{p/2}}\|_{L^p}^2 
\]
Hence
\[
\int_D \Big(\sum_{j=1}^\infty 
|S(t) \frac {e_j(x)}{\lambda_j^{q/2}}|^2\Big)^{\frac p2}dx
\le
\left(\sum_{j=1}^\infty   \lambda_j^{-q} e^{-2\lambda_j t} \|e_{j}\|_{L^p}^2\right)^{p/2}
\]
How to estimate $ \|e_{j}\|_{L^p}$? 
Again using the Sobolev embedding $H^{d(\frac 12-\frac 1p)}\subset
L^p$. Actually we are back again to Hilbert spaces and we obtain
nothing different with respect to our procedure which started in the Hilbert
spaces since the beginning. We leave the details to the reader.

Finally, let us point out that for $0<{\mathcal H}<\frac 12$, an
$L^p$-mild solution $z$ can be 
obtained in the Banach setting by means of Theorem 5.5 in \cite{BvNS}; this requires 
the operator $\Phi$ to be a $\gamma$-radonifying operator from 
$L^2_\sigma$ to $L^p_\sigma$, which 
is a quite strong assumption. Our method exploits the properties of the semigroup 
$S(t)$ so to allow weaker assumptions on the operator $\Phi$.
 \end{remark}

%%%%%%%%%%%%%%%%%%%%%
\section{Existence and uniqueness results}\label{sec-ex}
 In this section we study the Navier-Stokes initial problem \eqref{eq-abst} in the space
$L_{\sigma}^{p}$. 
We prove first the local existence result and  then the pathwise uniqueness.

%%%%%%%%%%%%%%%%%%%%%%%%%%%%%%%%%%%%%%%%%
\subsection{Local existence}
Following \cite{F}, we set $v=u-z$, where $z$ is the mild solution of the linear equation
\eqref{eq-z}. Therefore
\begin{equation}\label{eq-v}
\begin{cases}
\dfrac{dv}{dt}(t)+Av(t)=B(v(t)+z(t)) ,&\quad t>0\\
v(0)=u_0
\end{cases}
\end{equation}
and we get an existence result for $u$  by looking for an existence result for $v$.
This is given in the following theorem.

\begin{theorem} \label{th-ex}
Let $0<{\mathcal H}<1$, $d<p<\infty$ and $\Phi \in \mathcal L(L^2_\sigma,H^q)$ for some $q \in \mathbb R$.\\
Given $u_{0}\in L_{\sigma}^{p}$,
if the parameters fulfil
\begin{equation}\label{cond-dpq2}
 \frac d2(1-\frac 1p)-\frac q2<{\mathcal H}
\end{equation}
then there exists a  local  mild  $L^p$-solution to equation \eqref{eq-abst}.
\end{theorem}
\begin{proof} 
From Proposition \ref{pro-z} we know that $z$ has a version which belongs to $C([0,T];L^p_\sigma)$.

Now we observe that to find  a mild solution \eqref{mild} to equation \eqref{eq-abst} 
is equivalent to find a mild solution 
\[
v(t)= S(t)u_{0} +  \ \int_{0}^{t} S(t-s)B(v(s) +z(s))ds 
\]
to equation \eqref{eq-v}.

We work pathwise and  define a sequence by iterations: first $v^{0}=u_{0}$ and inductively
\[
v^{j+1}(t)= S(t)u_{0}+  \ \int_{0}^{t} S(t-s)B(z(s)+ v^{j}(s))  \ ds , \quad t \in [0,T]
\]
for $j=0,1,2,\ldots$.

Let us denote by $K_{0}$ the random constant 
\[ 
K_0=\max\left(\| u_{0}\|_{L^p_\sigma}, \sup_{t\in[0,T]} \|z(t) \|_{L^p_\sigma} \right).
\]
We shall show that there exists a random time $\tau>0$
such that  $\displaystyle\sup_{t\in [0,\tau]}\|v^{j}(t)\|_{L^p_\sigma}\le 2K_{0}$ 
for all $j\ge 1$. We have 
\[
 \|v^{j+1}(t)\|_{L^p_\sigma} \le 
\|S(t)u_{0}\|_{L^p_\sigma}  + \int_{0}^{t} \|S(t-s)B(v^{j}(s)+z(s))\|_{L^p_\sigma}  \ ds
\]
We observe that from \eqref{stimaSpr} and \eqref{stimaASp} we get
	\begin{equation}\label{one}
	\|S(t)u_{0}\|_{L^p_\sigma} \le  \|u_{0}\|_{L^p_\sigma}  
	\end{equation}
and 
\begin{equation}\label{two}
\begin{split}
\int_0^t &\|S(t-s)B((v^{j}(s)+z(s))\|_{L^p_\sigma} ds
\\&\le
 \int_{0}^{t}  \| A^{\frac{1}{2}} S(t-s) A^{-\frac{1}{2}} P \text{ div } ((v^{j}(s)+z(s))\otimes (v^{j}(s)+z(s))) \|_{L^p_\sigma}  \ ds, 
\\&	\le    
\int_{0}^{t}  \frac{1}{(t-s)^{\frac 12}}  \ \|S(t-s) A^{-\frac{1}{2}} P \text{ div }((v^{j}(s)+z(s))\otimes (v^{j}(s)+z(s)))\|_{L^p_\sigma}\ ds
\\&	\le
\int_{0}^{t}\frac{M}{(t-s)^{\frac 12 + \frac{d}{2p}}}\ \|A^{-\frac{1}{2}} P \text{ div } ((v^{j}(s)+z(s))\otimes (v^{j}(s)+z(s)))\|_{L^{p/2}_\sigma} \ ds
\\&	\le    
\int_{0}^{t}  \frac{M}{(t-s)^{\frac 12 + \frac{d}{2p}}} \ \| (v^{j}(s)+z(s))\otimes (v^{j}(s)+z(s)) \|_{L^{p/2}_\sigma}   \ ds
\\&	\le  
\int_{0}^{t}  \frac{M}{(t-s)^{\frac 12 + \frac{d}{2p}}} \  \|v^{j}(s)+z(s)\|_{L^p_\sigma}^{2}  \ ds
\end{split}
\end{equation}
From  (\ref{one}) and  (\ref{two})  we deduce that
\[
\begin{split}
\|v^{j+1}(t)\|_{L^p_\sigma} 
&\le  K_{0} +\int_{0}^{t} \frac{M}{(t-s)^{\frac 12+ \frac d{2p} }}\ \|v^{j}(s)+z(s)\|_{L^p_\sigma}^{2}  \ ds
\\&\le
 K_{0} +\int_{0}^{t} \frac{2M}{(t-s)^{\frac 12+ \frac d{2p} }}\ \|z(s)\|_{L^p_\sigma}^{2}  \ ds 
+\int_{0}^{t} \frac{2M}{(t-s)^{\frac 12+ \frac d{2p} }}\ \|v^{j}(s)\|_{L^p_\sigma}^{2}  \ ds
\end{split}
\]
Thus, when $\frac 12 +\frac d{2p}<1$ (i.e. $p>d$) we get
\[\begin{split}
 \sup_{t\in [0,T]}\|v^{j+1}(t)\|_{L^p_\sigma}   
&\le   K_{0}  + 2 M\   \frac{T^{\frac 12- \frac d{2p}}}{\frac 12- \frac d{2p}}   \  
\sup_{t\in [0,T]} \|z(t)\|_{L^p_\sigma}^{2} 
+  2 M\  \frac{T^{\frac 12- \frac d{2p}}}{\frac12- \frac d{2p}} \  (\sup_{t\in [0,T]}\|v^{j}(t)\|_{L^p_\sigma})^{2} 
\\
&\le K_{0}  +\frac{4pM}{p-d} T^{\frac 12- \frac d{2p}}   K_0^2
+  \frac{4pM}{p-d} \  T^{\frac 12- \frac d{2p}}\  (\sup_{t\in [0,T]}\|v^{j}(t)\|_{L^p_\sigma})^{2} 
\end{split}
\]

Now we show that if $\displaystyle\sup_{t\in [0,T]}\|v^{j}(t)\|_{L^p_\sigma}\le 2K_0$, then 
$\displaystyle\sup_{t\in [0,T]}\|v^{j+1}(t)\|_{L^p_\sigma}\le 2K_0$ on a suitable time interval.
Indeed, from the latter relationship we get
\[\begin{split}
\sup_{t\in [0,T]}\|v^{j+1}(t)\|_{L^p_\sigma}  
&\le K_0+ \frac{4pM}{p-d} T^{\frac 12- \frac d{2p}}K_0^2  
               +\frac{4pM}{p-d} T^{\frac 12- \frac d{2p}} 4K_0^2
\\&=2K_0\left(\frac 12 +\frac 12 \frac{20pM}{p-d} T^{\frac 12- \frac d{2p}} K_0\right).
\end{split}\]
Hence, when $T$ is such that 
\begin{equation*}
\frac{20pM}{p-d} T^{\frac 12- \frac d{2p}} K_0 \le 1
\end{equation*}
we obtain the required bound. 
Therefore we define the stopping time
\begin{equation}\label{cond-T}
\tau = \min\Big\{T,\left(\frac{p-d}{20pMK_0}  \right)^{\frac {2p}{p-d}}\Big\} 
\end{equation}
so that
\begin{equation}\label{cond-tau}
\frac{20pM}{p-d} \tau^{\frac 12- \frac d{2p}} K_0 \le 1
\end{equation}
and obtain that
\begin{equation}\label{stima-unif}
\sup_{t\in [0,\tau]}\|v^{j}(t)\|_{L^p_\sigma}\le 2K_0 \qquad \forall j .
\end{equation}

Now, we shall show the convergence of the sequence $v^{j}$.  
First, notice that 
\begin{multline*}
B(v^{j+1}+z)-B(v^j+z)
\\=-P\text{div}\ \big((v^{j+1}-v^j)\otimes v^{j+1}+v^j\otimes (v^{j+1}-v^j)
+(v^{j+1}-v^j)\otimes z+z\otimes (v^{j+1}-v^j)\big).
\end{multline*}
We proceed as in \eqref{two} and get
\[
\begin{split}
 \|v^{j+2}(t)-&v^{j+1}(t)\|_{L^p_\sigma} 
\\&\le
 \int_0^t \|S(t-s)\big(B(v^{j+1}(s)+z(s))-B(v^{j}(s)+z(s))\big)\|_{L^p_\sigma} ds
\\&\le 
\int_{0}^{t} \frac{M}{(t-s)^{\frac 12+\frac d{2p}}}  
 \big( \| v^{j}(s)\|_{L^p_\sigma} + \| v^{j+1}(s)\|_{L^p_\sigma}+2\|z(s)\|_{L^p_\sigma}  \big) \ 
\|v^{j+1}(s)-v^{j}(s) \|_{L^p_\sigma} ds 
\end{split}\]
Hence, using \eqref{stima-unif}   we get 
\[\begin{split}
\sup_{t \in [0,\tau]}\|v^{j+2}(t)-v^{j+1}(t)\|_{L^p_\sigma} 
&\le 
\int_{0}^{\tau} \frac{ M 6K_0}{(t-s)^{\frac 12+\frac d{2p}}}  ds \ 
\left(\sup_{s \in [0,\tau]} \|v^{j+1}(s)-v^{j}(s) \|_{L^p_\sigma} \right)
\\
&=  \frac{12pM  K_{0}}{p-d} \  \tau ^{\frac 12- \frac d{2p}}  \  \left( \sup_{t\in [0,\tau]}\|v^{j+1}(t)-v^{j}(t) \|_{L^p_\sigma} \right)
\end{split}\]
Setting $C_0= \frac{12pM  K_{0}}{p-d} \  \tau ^{\frac 12- \frac d{2p}}$, 
from \eqref{cond-T}-\eqref{cond-tau} we obtain that  $C_0<1$. 
Moreover
\[\begin{split}
\sup_{t \in [0,\tau]}\|v^{j+2}(t)-v^{j+1}(t)\|_{L^p_\sigma} 
&\le C_0 \sup_{t\in [0,\tau]}\|v^{j+1}(t)-v^{j}(t) \|_{L^p_\sigma} 
\\&
\le
C_0^{j+1} \sup_{t\in [0,\tau]}\|v^{1}(t)-v^{0}(t) \|_{L^p_\sigma} 
\end{split}\]
Therefore $\{v^j\}_j$ is a Cauchy sequence; hence it converges, that is 
there exists $v \in C([0,\tau]; L^{p}_\sigma)$ such that 
 $v^{j}\rightarrow v$ in $C([0,\tau]; L^{p}_\sigma)$.  
This proves the existence of  a unique local mild $L^p$-solution $v$ for equation 
 \eqref{eq-v}.

Since $u=v+z$, we have got a local mild $L^p$-solution $u$ for equation \eqref{eq-abst}. 
\end{proof}

%%%%%%%%%REMARK
\begin{remark}
We briefly discuss the case of 
 cylindrical noise, i.e. $\Phi= Id$. Bearing in mind Theorem \ref{th-ex}, the
parameters fulfil
\begin{equation}\label{q0}
\frac d2(1-\frac 1p)<{\mathcal H}<1.
\end{equation}
When $2=d<p$, this means that $p$ and ${\mathcal H}$ must be chosen in such a way that 
\begin{equation}
1-\frac 1p<{\mathcal H}<1
\end{equation}
This means that ${\mathcal H}$ must be at least larger than $\frac 12$.
On the other hand,  when $3=d<p$ we cannot apply our procedure, since  $\frac d2(1-\frac 1p)>1$  and therefore 
 the set of conditions \eqref{q0} is void.
\end{remark}

%%%%%%%%%%%%%%%%%%%%%%
\subsection{Uniqueness}
Now we show pathwise uniqueness of the solution given in Theorem \ref{th-ex}. 
%TEO
\begin{theorem} 
Let $d<p<\infty$ and $\Phi \in \mathcal L(L^2_\sigma,H^q)$ for some $q \in \mathbb R$.\\
Given $u_{0}\in L_{\sigma}^{p}$,
if
\[
\frac d2 (1-\frac 1p)-\frac q2<{\mathcal H}<1
\]
then the   local  mild  $L^p$-solution  to equation \eqref{eq-abst} given in Theorem \ref{th-ex} is pathwise unique.
\end{theorem}
\begin{proof}
Let $u$ and $\tilde u$ be two mild solutions of equation \eqref{eq-abst} with the same fBm  and the same initial velocity. Their difference satisfies an equation where the noise has disappeared. Hence we work pathwise.
We get
\[
 u(t)-\tilde u(t)= \ \int_{0}^{t} S(t-s) \big(B(u(s))-B(\tilde u(s)) \big)  \ ds .  
\]
 
Writing $B(u)-B(\tilde u)=B(u-\tilde u,u)+B(\tilde u, u-\tilde u)$,
by classical estimations as before we have 

\[
\begin{split}
\| u(t)-\tilde u(t)\|_{L^p_\sigma}
&\le \ \int_{0}^{t} \|S(t-s) \big(B(u(s))-B(\tilde u(s)) \big)\|_{L^p_\sigma}  \ ds 
\\&
\le  \int_{0}^{t} \frac{M}{(t-s)^{\frac{1}{2}+ \frac{d}{2p}}} (\| u(s)\|_{L^p_\sigma}+ \| \tilde u(s)\|_{L^p_\sigma})  \|u(s)-\tilde u(s)\|_{L^p_\sigma}  \ ds   
\end{split}
\]
Thus 
\[
\sup_{[0,\tau]}\|u(t)-\tilde u(t)\|_{L^p_\sigma}
\le  
4K_0 M  \frac{ \tau^{\frac{1}{2}- \frac{d}{2p}}}{\frac{1}{2}- \frac{d}{2p}} \  
\sup_{t\in [0,\tau]}\|u(t)-\tilde u(t)\|_{L^p_\sigma}  .
\]
Keeping in mind the definition \eqref{cond-T} of $\tau$ and \eqref{cond-tau} we get
\[
\sup_{[0,\tau]}\|u(t)-\tilde u(t)\|_{L^p_\sigma}
\le  
\frac 25
\sup_{[0,\tau]}\|u(t)-\tilde u(t)\|_{L^p_\sigma}
\]
which implies   $u(t)=\tilde u(t)$ for any $t \in [0,\tau]$. 
\end{proof}

%%%%%%%%%%%%%%%%%%
\subsection{Global existence} 
Let us recall that \cite{Fang} proved global existence an uniqueness of
a $L^4((0,T)\times D)$-valued solution.
A similar result of global existence for a less regular (in time) 
solution holds in our setting.

Let us begin with the case $d=2$ and
consider a process solving equation \eqref{eq-abst} whose paths are in 
$L^{\frac {2p}{p-2}}(0,T;L^p_\sigma)$. 
Its local existence comes from the previous results. 
However we can prove an a priori bound leading to global existence.

Let us multiply equation \eqref{eq-v} by $v$ in $L^2_\sigma$; we obtain 
by classical techniques (see Lemma 4.1 of \cite{Fla}) 
\[\begin{split}
\frac 12 \frac d{dt}\|v(t)\|_{L^2_\sigma}^2+\|\nabla v(t)\|^2_{L^2}&=\langle B(v(t)+z(t),z(t)),v(t)\rangle
\\&
\le\|v(t)+z(t)\|_{L^4_\sigma} \|z(t)\|_{L^4_\sigma} \|\nabla v(t)\|_{L^2}
\\&
\le\frac 12 \|\nabla v(t)\|^2_{L^2}
+\frac C2  \|z(t)\|_{L^4_\sigma}^4 \|v(t)\|_{L^2_\sigma}^2+\frac C2  \|z(t)\|_{L^4_\sigma}^4
\end{split}
\]
Hence 
\[
\frac d{dt}\|v(t)\|_{L^2_\sigma}^2\le C  \|z(t)\|_{L^4_\sigma}^4
\|v(t)\|_{L^2_\sigma}^2+C  \|z(t)\|_{L^4_\sigma}^4 .
\]
As soon as $z$ is a $C([0,T];L^4_\sigma)$-valued process  we get by means of Gronwall lemma that
$v \in L^\infty(0,T;L^2_\sigma)$. And integrating in time the first inequality we also obtain that
$v\in L^2(0,T;H^1)$.
By interpolation $L^\infty(0,T;L^2_\sigma)\cap L^2(0,T;H^1)\subset
L^{\frac {2p}{p-2}}(0,T;H^{1-\frac 2p})$ for $2<p<\infty$.
Using the Sobolev embedding $H^{1-\frac 2p}\subset L^{p}_\sigma$,
we have the a priori estimate for $v$ in the $L^{\frac {2p}{p-2}}(0,T;L^p_\sigma)$ norm, 
which provides the global existence of $v$ and hence of $u$. This holds for $d=2$ and $4\le p<\infty$,
since the global estimate holds when $z$ is $C([0,T];L^4_\sigma)$-valued at least.
\\Notice that for $d=2$ and  $p=4$ we obtain the same result as by Fang,   Sundar and   Viens
 (see Corollary 4.3 in \cite{Fang}). 

Similarly one proceeds when $d=3$. The change is in the Sobolev embedding, 
which depends on the spatial dimension. Thus from $v \in L^\infty(0,T;L^2_\sigma)\cap
L^2(0,T;H^1)$ we get by interpolation that 
$v \in L^{\frac{4p}{3(p-2)}}(0,T;H^{3\frac{p-2}{2p}})$ for $2<p\le 6$. 
Using the Sobolev embedding
$H^{3\frac{p-2}{2p}}\subset L^{p}_\sigma$ we conclude that the 
$L^{\frac{4p}{3(p-2)}}(0,T;L^{p}_\sigma)$-norm of $v$ is bounded. 
Hence the global existence of a solution $v \in L^{\frac{4p}{3(p-2)}}(0,T;L^{p}_\sigma)$ for $4\le p\le 6$ 
 as well as of a solution $u  \in L^{\frac{4p}{3(p-2)}}(0,T;L^{p}_\sigma)$.

%%%%%%%%%%%%%%%%%%%%%%%%
\section*{Acknowledgements}
%%%%%%%%%%%%%%%%%%%%%%%%
C. Olivera  is partially supported by FAPESP 
		by the grants 2017/17670-0 and 2015/07278-0.
B. Ferrario is  partially supported by INdAM-GNAMPA, 
 by PRIN 2015 
''Deterministic and stochastic evolution equations'' and by
 MIUR -Dipartimenti di Eccellenza Program (2018-2022) - Dept. of Mathematics ''F. Casorati'', University of Pavia.

%%%%%%%%%%%%%%%%%


\begin{thebibliography}{99}
%%%%%%%%%%%%%%%%%%

\bibitem{Bene}
Albeverio S, Ferrario B (2008)  
{\it Some Methods of Infinite Dimensional Analysis in Hydrodynamics: An Introduction},
In ''SPDE in Hydrodynamic: Recent Progress and Prospects'',  
Lecture Notes in Mathematics 1942:  Springer, 1-50.

\bibitem{BT}
Bensoussan A, Temam R (1973)
\'Equations stochastiques du type Navier-Stokes.
\textit{J. Functional Analysis} 13: 195-222

\bibitem{BvNS}
Brze\'zniak Z, van Neerven J, Salopek D (2012)
Stochastic evolution equations driven by Liouville fractional Brownian motion. 
\textit{Czechoslovak Math. J.} 62   no. 1: 1-27.

\bibitem{Co}
\v{C}oupek P, Maslowski B,  Ondrej\'at M (2018)
 $L^{p}$-valued stochastic convolution integral driven by Volterra noise.
\textit{Stoch. Dyn.} 18 no. 6, (published online 11 January 2018) https://doi.org/10.1142/S021949371850048X

\bibitem{DPDM2002}
Duncan T E,  Pasik-Duncan B, Maslowski B (2002)
Fractional Brownian motion and stochastic equations in Hilbert spaces.
\textit{Stoch. Dyn.} 2 no. 2: 225-250.

\bibitem{Fang}
Fang L,  Sundar P,  Viens F (2013)
Two-dimensional stochastic Navier-Stokes equations with fractional Brownian noise.
\textit{Random Operators and Stochastic Equations}   21: 135-158.

\bibitem{F}
Flandoli F (1994)
Dissipativity and invariant measures for stochastic Navier-Stokes equations.
\textit{NoDEA Nonlinear Differential Equations Appl.} 1: 403-423.

\bibitem{Fla}
Flandoli F (2008)
\textit{An introduction to $3D$ stochastic fluid dynamics.}  
In SPDE in Hydrodynamic: Recent
Progress and Prospects,  Lecture Notes in Math. 1942. Springer,
Berlin: 51-150.

\bibitem{Kato2}
Fujita  H,  Kato T (1964)
On the Navier-Stokes initial value problem. I.
\textit{  Arch. Rational Mech. Anal.} 16: 269-315.

\bibitem{Giga2}
Giga Y (1986)
Solutions for semilinear parabolic equations in $L^p$ and regularity
of weak solutions of the Navier-Stokes system. 
\textit{ J. Differential Equations} 62: 186-212.

\bibitem{Giga}
Giga Y, Miyakawa T (1985)
Solutions in $L_r$ of the Navier-Stokes initial value problem.
\textit{Arch. Rational Mech. Anal.} 89: 267-281.

\bibitem{Kato}
Kato T,  Fujita H (1962)
On the nonstationary Navier-Stokes system.
\textit{  Rend. Sem. Mat. Univ. Padova} 32: 243-260.

\bibitem{Kato3}
Kato T (1984)  
Strong $L^p$-solutions of the Navier-Stokes equation in $\R^{m}$, with
applications to weak solutions.
\textit{ Math. Z.} 187: 471-480.

\bibitem{K}
Kuksin S B, Shirikyan A (2012)
\textit{Mathematics of two-dimensional turbulence}. 
Cambridge Tracts in Mathematics, 194. Cambridge University Press, Cambridge.

\bibitem{Lema}
Lemarie-Rieusset P G (2002)
\textit{ Recent developments in the Navier-Stokes problem}, Chapman  Hall/CRC.

\bibitem{Leray}  
Leray J (1934)
Sur le mouvement d'un liquide visqueux emplissant l'espace.
\textit{ Acta Math.} 6:  3193-248.

\bibitem{lion1}
Lions P-L (1996)
\textit{ Mathematical topics in fluid mechanics, Vol. I: incompressible models}, 
Oxford Lecture Series in Mathematics and its applications, 3, Oxford University Press.

\bibitem{N}
Nualart D (2006)
\textit{ Malliavin Calculus and Related Topics}, Second Edition. Springer New York.

\bibitem{PDDM2006}
Pasik-Duncan B, Duncan T E,  Maslowski B (2006)
\textit{ Linear stochastic equations in a Hilbert space with a fractional Brownian motion}. 
Stochastic processes, optimization, and control theory: applications in financial engineering, 
queueing networks, and manufacturing systems, 201-221, Internat. Ser. Oper. Res. Management Sci., 
94, Springer, New York.

\bibitem{Soh}
Sohr H (2001) 
\textit{The Navier-Stokes equations. An elementary functional analytic approach},
Birkh\"auser Advanced Texts, Birkh\"auser Verlag, Basel. 

\bibitem{vN}
van Neerven J (2010)
$\gamma$-radonifying operators - a survey. 
The AMSI-ANU Workshop on Spectral Theory and Harmonic Analysis, 1-61, 
Proc. Centre Math. Appl. Austral. Nat. Univ., 44, Austral. Nat. Univ., Canberra.

\bibitem{VF} 
Vishik M J, Fursikov A V (1980)
\textit{ Mathematical Problems of Statistical Hydromechanics },  1
Kluver Dordreeht.

\bibitem{W}
Weissler F B (1980)
The Navier-Stokes initial value problem in L$^p$.
\textit{Arch. Rational Mech. Anal.} 74 no. 3: 219-230.

\end{thebibliography}
\end{document}